\documentclass[a4paper]{amsart}
\usepackage{amsmath}
\usepackage{amssymb}
\usepackage{enumitem}
\usepackage{layouts}
\usepackage{amsthm}
\usepackage{color}

\newcommand{\R}{\mathbb{R}}
\newcommand{\C}{\mathbb{C}}
\newcommand{\brk}[1]{\langle #1 \rangle}
\newcommand{\bbrk}[1]{\langle\langle #1 \rangle\rangle}

\newtheorem{theorem}{Theorem}[section]
\newtheorem{definition}{Definition}[section]
\newtheorem{lemma}[theorem]{Lemma}
\newtheorem{prop}[theorem]{Proposition}

\numberwithin{equation}{section}

\title{Morse index of free boundary disk in pseudoconvex domain}
\author{Chi Fai Chau}

\begin{document}
\maketitle
\begin{abstract}
In this paper we study the Morse index for the $\overline{\partial}$-energy of a non-holomorphic disk in a strictly pseudoconvex domain in $\C^n$ or in a K\"ahler manifold with non-negative bisectional curvature. We give a 
proof that a $\overline{\partial}$-energy minimizing disk is holomorphic; in fact, more generally we show that
a non-holomorphic critical disk for the $\overline{\partial}$-energy has Morse index at least $n-1$.
We also extend the result to domains which satisfy the weaker $k$-pseudoconvexity property for $k\geq 2$.
\end{abstract}

\section{Introduction}
In this paper, we study the Morse index of a free boundary surface in a pseudoconvex domain in $\C^n$. We consider those maps, which we
call $\overline{\partial}$-harmonic, that are critical points of the $\overline{\partial}$-energy on the space of maps from unit disk $D$ to a pseudoconvex domain $N$ such that the boundary $\partial D$ is mapped into the boundary $\partial N$. As an application, we show that any minimizing $\overline{\partial}$-harmonic map from the disk to a strictly pseudoconvex domain is holomorphic. Combined with existence theory, the study of the Morse index can provide a powerful tool in studying the topology of pseudoconvex domain in $\C^n$. We also extend the result to any strictly pseudoconvex domain in a 
K\"{a}hler manifold with nonnegative bisectional curvature. In \cite{SiuYau}, Siu-Yau proves the Frankel conjecture, which states any compact K\"{a}hler manifold with positive bisectional curvature is biholomorphic to complex projective space. Their holomorphicity theorem plays an important role in the proof.
\begin{theorem} (\cite{SiuYau})
Suppose $M$ is compact K\"{a}hler manifold with positive holomorphic bisectional curvature. Let $f:\mathbb{P}^1\to M$ be an energy-minnimizing map. Then $f$ is either holomorphic or anti holomorphic.
\end{theorem}
We remark that in the K\"ahler case there are three different energies, the $\overline{\partial}$-energy, the $\partial$-energy and the full energy. For 
closed surfaces or surfaces with a fixed boundary curve, these energies are equivalent and have the same critical points with the same Morse indices. The
situation for the free boundary condition is quite different. In this case, the critical points, when the domain is the unit disk, are always minimal
surfaces, but the boundary conditions are different. For the full energy, the boundary condition is the usual free boundary condition which says that
the surface meets the boundary orthogonally (the outer unit normal of the surface is normal to the boundary of the domain). For the 
$\overline{\partial}$-energy, the boundary condition says that sum of the outer normal to the surface plus the complex structure $J$ applied to the
tangent vector to the boundary is parallel to the normal to the boundary of the domain (see Section 2.1). For the $\partial$-energy the boundary condition would have the sum replaced by the difference. 

We extend the holomorphicity theorem to a free boundary disk in a pseudoconvex domain in terms of the Morse index of $\overline{\partial}$-harmonic
map. The main result of this paper is the following.
\begin{theorem}\label{maintheorem}
Let $D$ be a unit disk and $N$ be a strictly pseudoconvex domain in $ C^n$. Let $f:D\to N$ be a smooth map, with $f(\partial D)\subset \partial N$. Suppose $f$ is critical point of the $\overline{\partial}$-energy. If $f$ is not holomorphic, then $f$ has Morse index at least $n-1$ for the $\overline{\partial}$-energy. In particular, if $f$ is stable for the 
$\overline{\partial}$-energy, then $f$ is holomorphic.
\end{theorem}
The key point in our proof is the construction of holomorphic variations that give negative values in index form. The holomorphic variations then gives a lower bound on the index. In the second section, we derive the complex version of the second variation formular of $\overline{\partial}$-energy and give the condition on being an admissable variation. We then construct holomorphic variations and prove the main result in Section 3. In \cite{Fraser2}, Fraser gave the index estimate and instability theorem for minimal disk in $k$-convex hypersurface in Riemannian manifold. We generalize the concept of pseudoconvexity and define $k$-pseudoconvexity that is analogous to $k$-convexity in real manifolds in the last section. We then give index theorem for the weaker boundary condition $k$-pseudoconvex  and also in the case of pseudoconvex domain in K\"{a}hler manifold with positive bisectional curvature since the construction of holomorphic variations can also be adpated in such cases.

To illustrate the theorem and show that the hypotheses are necessary, we give several examples. First we give an
example of a non-holomorphic critical disk for the $\overline{\partial}$-energy in a strictly pseudoconvex domain. We
also give an example of a stable critical disk for the full energy (free boundary disk) in a strictly pseudoconvex
domain which is not holomorphic. Further we give an example of a holomorphic free boundary disk in a strictly
pseudoconvex domain which is unstable for the full energy. Finally we give an example of a critical disk for the
$\overline{\partial}$-energy in a weakly pseudoconvex domain which is stable. This shows that the strict 
pseudoconvexity assumption is necessary in Theorem \ref{maintheorem}. 

We also define the notion of a strictly $k$-pseudoconvex domain for integer $k\geq 1$. This means that the trace
of the second fundamental form on any $k$-dimensional complex subspace tangent to the
boundary is positive at each point of the boundary. 
\begin{theorem}
Let $D$ be the unit disk and $N$ be a strictly $k$-pseudoconvex domain in $ C^n$. Let $f:D\to N$ be a smooth map, with $f(\partial D)\subset \partial N$. Suppose $f$ is a critical point of the $\overline{\partial}$-energy. If $f$ is not holomorphic, then $f$ has Morse index at least $n-k$ for the $\overline{\partial}$-energy.
\end{theorem}
For the case $k=1$, the domain is strictly pseudoconvex and the above theorem coincides with theorem \ref{maintheorem}.

\subsection*{Acknowledgement} I would like to thank my advisor Richard Schoen for suggesting this topic and all the support and encouragement throughout the progress of this work.

\section{Preliminaries}

\begin{definition}{(Levi Pseudoconvex)}
Let $\Omega\subset\subset \C^n$ be a domain with $C^2$ boundary. Let $\rho:\C^n\to\R$ be defining function for $\Omega$. The domain $\Omega=\{z\in\C^n : \rho(z)<0\}$ is strictly Levi pseudoconvex if for every point $p\in\partial\Omega$
\[\sum_{j,k=1}^n\frac{d^2 \rho}{dz_jd\overline{z_k}}(p)\omega_j\overline{\omega_k}>0\]
for all $\omega\in C^n$ that satisfy
\[\sum_{j=1}^n\frac{d\rho}{dz_j}\omega_j=0\]
\end{definition}
It is equivalent to saying that $\langle \nabla_V \overline{V},\nu\rangle<0$ for any $V$ in holomorphic tangent bundle on $\partial\Omega$ that satisfies $V(\rho)=0$, where $\nu$ is outward unit normal to $\partial\Omega$

Now suppose $D$ is a unit disk in $\C$ and $N$ is a strictly pseudoconvex domain in $\C^n$ and their metric are $g$ and $h$ respectively. Let $f:D\to N$ be a smooth map. We define the partial energies as follows.

\begin{definition}{($\partial$-energy and $\overline{\partial}$-energy)}
In local holomorphic coordinate $\{f^1,\ldots, f^n\}$ on $N$ and $\{\omega\}$ on $D$, the pointwise $\overline{\partial}$-energy of the map $f$, $|\overline{\partial}f|^2$,  is defined by
\[g^{-1}\frac{df^\alpha}{d\bar{\omega}}\overline{\frac{df^\beta}{d\bar{\omega}}}h_{\alpha\bar{\beta}}\]
and we define the $\overline{\partial}$-energy $E''(f)$ to be $\int_D |\overline{\partial}f|^2 d\mathcal{A}$
where $d\mathcal{A}$ is volume form of $\Sigma$. It can be observed that $f$ is holomorphic if and only if $E''(f)=0$. Similarly, the pointwise $\partial$-energy $E'(f)$ of the map $f$, $|\partial f|^2$, is defined by
\[ g^{-1}\frac{df^\alpha}{d\omega}\overline{\frac{df^\beta}{d\omega}}h_{\alpha\bar{\beta}}\]
and the $\partial$-energy $E'(f)$ is defined by $\int_D |\partial f|^2\;d\mathcal{A}$.
\end{definition}
From the definition, the energy $E(f)$ of the map $f$ is, therefore, equal to $E'(f)+E''(f)$. The pullback of the K\"{a}hler form of $N$ under $f$ is
\[\sqrt{-1}h_{\alpha\beta}df^\alpha\wedge d\overline{f^\beta}=\sqrt{-1}h_{\alpha\beta}(\frac{df^\alpha}{d\omega}\overline{\frac{df^\beta}{d\omega}}-\frac{df^\alpha}{d\bar{\omega}}\overline{\frac{df^\beta}{d\bar{\omega}}})d\omega\wedge d\bar{\omega}\]
and hence, the difference of the partial energies,
\[E'(f)-E''(f)=\int_D \sqrt{-1}h_{\alpha\beta}df^\alpha\wedge d\overline{f^\beta}=\int_D f^\ast\omega^N\]
A. Lichnerowicz proves that the difference of the partial energies is homotopy invariant on compact manifolds (without boundary). In fact, If $f_t:D\to N$ is a family of smooth maps, then $\frac{d}{dt}f_t^\ast\omega^N=d(f_t^\ast i(v)\omega)$ where $v=\frac{df}{dt}$. Therefore, the $E'(f)-E''(f)$ is constant when $f_t$ have fixed boundary.

\subsection{The First and Second Variation Formula}

Let $f_t:D\to N$ be a family of smooth map, with $f(\partial D)\subset\partial N$. Let $V=\frac{df}{dt}|_{t=0}\in\Gamma(f^\ast TN)$ be the variation vector field and $z=x+iy$ be holomorphic coordinate on $D$. The first variation of $\overline{\partial}$-energy of $f$ with real variation field $V$ is given by 
\[\frac{d}{dt}|_{t=0}E''(f)=\int_D \brk{V,\nabla_{\frac{d}{d\bar{z}}} \frac{df}{dz}}\;dx\wedge dy+\int_{\partial D} \brk{V,\frac{df}{dr}+J\frac{df}{d\theta}}\;d\theta\]
We say that $f$ is harmonic if $\nabla_{\frac{d}{d\bar{z}}} \frac{df}{dz}=0$. Then for any harmonic map $f$, it is critical point of $\overline{\partial}$-energy if and only if $\frac{df}{dr}+J\frac{df}{d\theta}=\lambda\nu$ for some function $\lambda$, where $\nu$ is outward normal to $\partial N$.

The second variation of $\overline{\partial}$-energy along the smooth variation $f_t$ with variation field $V$ gives a symmetric billiear index form
\begin{equation}\label{2ndvar}
\begin{split}
I(V,V)=\frac{1}{2}\Big[\int_D &\|\nabla V\|^2-\langle R(V),V\rangle dx\wedge dy+\int_{\partial D}\langle\nabla_V V,\frac{df}{dr}+J\frac{df}{d\theta}\rangle d\theta\\
&+\int_{\partial D}\langle J\nabla_{\frac{d}{d\theta}}V, V\rangle\;d\theta\Big]
\end{split}
\end{equation}
The norm of $\nabla V$ is given by
\[\|\nabla V\|^2=\langle\nabla_{\frac{d}{dx}}V,\nabla_{\frac{d}{dx}}V\rangle+\langle\nabla_{\frac{d}{dy}}V, \nabla_{\frac{d}{dy}}V\rangle\]
and the endomorphism $R$ on vector bundle $f^\ast TN$ is defined by
\[R(V)=R(V,\frac{df}{dx})\frac{df}{dx}+R(V,\frac{df}{dy})\frac{df}{dy}\]

The index of $\overline{\partial}$-energy of $f$ is defined by the maximal dimension of subspace of $\Gamma(f^\ast TN)$ in which the index form in \eqref{2ndvar} is negative definite. Now consider the complexified pull back tangent bundle $f^\ast TN \otimes \C$. The metric extends to $f^\ast TN \otimes \C$ as a complex billinear form $(\;,\;)$ or a Hermitian inner product $\langle\langle\;,\;\rangle\rangle$. And the connection $\nabla$ extends to a complex linear connection on  $f^\ast TN \otimes \C$. In complexified pull back tangent bundle, the vector field $\frac{df}{d\bar{z}}$ will be a useful indicator in determining whether a section $V\in\Gamma(f^\ast T^{(1,0)}N)$ is a variation vector field to a harmonic map $f$ that is critical to $\overline{\partial}$-energy.
\begin{prop}\label{varfield}
Let $V\in\Gamma(f^\ast TN)$, $f:D\to N$ be a map which is critical to $\overline{\partial}$-energy. Let $W=V-iJV$. Then $V$, $JV$ are variation vector field along deformation of $f$ if and only if $\bbrk{W,\frac{df}{d\bar{z}}}=0$.
\end{prop}

\begin{proof}
On the boundary $\partial D$, $\frac{df}{d\bar{z}}=\alpha(\frac{df}{dr}+i\frac{df}{d\theta})$ for some complex valued function $\alpha$. Then
\begin{align*}
(\frac{df}{dr}+i\frac{df}{d\theta})^{(1,0)}&=\frac{1}{2}\Big[(\frac{df}{dr}+J\frac{df}{d\theta})-iJ(\frac{df}{dr}+J\frac{df}{d\theta})\Big]\\
&=\frac{\lambda}{2}(\nu-iJ\nu)
\end{align*}
In the last equality, we use the fact that $\frac{df}{dr}+J\frac{df}{d\theta}=\lambda\nu$ for some function $\lambda$ as $f$ is critical point to $\overline{\partial}$-energy. Then
\begin{align*}
\bbrk{W,(\frac{df}{dr}+i\frac{df}{d\theta})}&=\bbrk{W,(\frac{df}{dr}+i\frac{df}{d\theta})^{(1,0)}}\\
&=\bbrk{V-iJV, \frac{\lambda}{2}(\nu-iJ\nu)}\\
&=\lambda[\brk{V,\nu}-i\brk{JV,\nu}]
\end{align*}
Therefore, $V$, $JV$ are variation vector fields if and only if $\bbrk{W,\frac{df}{d\bar{z}}}=0$
\end{proof}

The index form $I(V,V)$ extends to a Hermitian symmetric bilinear form on $f^\ast TN \otimes \C$,
\begin{align*}
I(V,V)=\frac{1}{2}\Big[\int_\Sigma &\|\nabla V\|^2-\langle\langle R(V),V\rangle\rangle d\mathcal{A}+\int_{\partial \Sigma}\langle\langle\nabla_V \overline{V},\frac{df}{dr}+J\frac{df}{d\theta}\rangle\rangle d\theta\\
&+\int_{\partial \Sigma}\langle\langle J\nabla_{\frac{d}{d\theta}}V, V\rangle\rangle\;d\theta\Big].
\end{align*}
Following the calculation in \cite{Micallef},
\begin{lemma}\label{indexformula}
If $V$ is a variation vector field to $f$, then
\begin{align*}
I(V,V)=2\int_D&\|\nabla_{\frac{d}{d\bar{z}}}V\|^2-\langle\brk{R(V,\frac{df}{dz})\frac{df}{d\bar{z}},V}\rangle dx\wedge dy+\frac{1}{2}\int_{\partial D}\bbrk{\nabla_V \overline{V}, \frac{df}{dr}+J\frac{df}{d\theta}}d\theta\\
&-\frac{i}{2}\int_{\partial D}\langle\brk{\nabla_{\frac{d}{d\theta}}V+iJV,V}\rangle\;d\theta
\end{align*}
\end{lemma}

\begin{proof}
By direct computation,
\begin{align*}
\|\nabla_{\frac{d}{dz}}V\|^2+\|\nabla_{\frac{d}{d\bar{z}}}V\|^2=\frac{1}{2}\Big[\|\nabla_{\frac{d}{dx}}V\|^2+\|\nabla_{\frac{d}{dy}}V\|^2\Big]
\end{align*}
and
\begin{align*}
&\bbrk{R(V,\frac{df}{dz})\frac{df}{d\bar{z}},V}+\bbrk{R(V,\frac{df}{d\bar{z}})\frac{df}{dz},V}\\
=&\frac{1}{2}\Big[\bbrk{R(V,\frac{df}{dx})\frac{df}{dx},V}+\bbrk{R(V,\frac{df}{dy})\frac{df}{dy},V}\Big].
\end{align*}
Then we have
\begin{equation}\label{2ndvar2}
\begin{split}
I((V,V)=&\int_{D}\|\nabla_{\frac{d}{dz}}V\|^2+\|\nabla_{\frac{d}{d\bar{z}}}V\|^2-\bbrk{R(V,\frac{df}{dz})\frac{df}{d\bar{z}},V}\\
&\qquad-\bbrk{R(V,\frac{df}{d\bar{z}})\frac{df}{dz},V} dx\wedge dy+\frac{1}{2}\int_{\partial D}\bbrk{\nabla_V \overline{V},\frac{df}{dr}+J\frac{df}{d\theta}}d\theta\\
&\qquad+\frac{1}{2}\int_{\partial D}\bbrk{J\nabla_{\frac{d}{d\theta}}V,V}d\theta
\end{split}
\end{equation}
As we have
\begin{align*}
\|\nabla_{\frac{d}{dz}}V\|^2&=(\nabla_{\frac{d}{dz}}V,\nabla_{\frac{d}{d\bar{z}}}\overline{V})\\
&=\frac{d}{d\bar{z}}(\nabla_{\frac{d}{dz}}V,\overline{V})-(\nabla_{\frac{d}{d\bar{z}}}\nabla_{\frac{d}{dz}}V,\overline{V})\\
&=\frac{d}{d\bar{z}}(\nabla_{\frac{d}{dz}}V,\overline{V})-(\nabla_{\frac{d}{dz}}\nabla_{\frac{d}{d\bar{z}}}V,\overline{V})-(R(\frac{df}{d\bar{z}},\frac{df}{dz})V,\overline{V})\\
&=\frac{d}{d\bar{z}}(\nabla_{\frac{d}{dz}}V,\overline{V})-\frac{d}{dz}(\nabla_{\frac{d}{d\bar{z}}}V,\overline{V})+\|\nabla_{\frac{d}{d\bar{z}}}V\|^2-\bbrk{R(\frac{df}{d\bar{z}},\frac{df}{dz})V,\overline{V}},
\end{align*}
using this in \eqref{2ndvar2}, we get
\begin{align*}
I((V,V)=&\int_{D}2\|\nabla_{\frac{d}{d\bar{z}}}V\|^2+\frac{d}{d\bar{z}}(\nabla_{\frac{d}{dz}}V,\overline{V})-\frac{d}{dz}(\nabla_{\frac{d}{d\bar{z}}}V,\overline{V})\\
&\qquad-\bbrk{R(\frac{df}{d\bar{z}},\frac{df}{dz})V,V}-\bbrk{R(V,\frac{df}{dz})\frac{df}{d\bar{z}},V}\\
&\qquad-\bbrk{R(V,\frac{df}{d\bar{z}})\frac{df}{dz},V} dx\wedge dy+\frac{1}{2}\int_{\partial D}\bbrk{\nabla_V \overline{V},\frac{df}{dr}+J\frac{df}{d\theta}}d\theta\\
&\qquad+\frac{1}{2}\int_{\partial D}\bbrk{J\nabla_{\frac{d}{d\theta}}V,V}d\theta.
\end{align*}
By Stoke's theorem,
\[\int_{D}\frac{d}{d\bar{z}}(\nabla_{\frac{d}{dz}}V,\overline{V})-\frac{d}{dz}(\nabla_{\frac{d}{d\bar{z}}}V,\overline{V})=-\frac{i}{2}\int_{\partial D}\bbrk{\nabla_{\frac{d}{d\theta}}V,V}\;d\theta\]
and Bianchi identities
\[-\bbrk{R(\frac{df}{d\bar{z}},\frac{df}{dz})V,V}-\bbrk{R(V,\frac{df}{d\bar{z}})\frac{df}{dz},V}=\bbrk{R(\frac{df}{dz},V)\frac{df}{d\bar{z}},V},\]
we finally have the equality
\begin{align*}
I(V,V)&=2\int_{D}\|\nabla_{\frac{d}{d\bar{z}}}V\|^2-\bbrk{R(V,\frac{df}{dz})\frac{df}{d\bar{z}},V}dx\wedge dy\\
&\qquad+\frac{1}{2}\int_{\partial D}\bbrk{\nabla_V \overline{V},\frac{df}{dr}+J\frac{df}{d\theta}}\;d\theta+\frac{1}{2}\int_{\partial D}\bbrk{J\nabla_{\frac{d}{d\theta}}V,V}\;d\theta\\
&\qquad-\frac{i}{2}\int_{\partial D}\bbrk{\nabla_{\frac{d}{d\theta}}V,V}\;d\theta\\
&=2\int_D\|\nabla_{\frac{d}{d\bar{z}}}V\|^2-\langle\brk{R(V,\frac{df}{dz})\frac{df}{d\bar{z}},V}\rangle dx\wedge dy+\frac{1}{2}\int_{\partial D}\bbrk{\nabla_V \overline{V}, \frac{df}{dr}+J\frac{df}{d\theta}}d\theta\\
&\qquad-\frac{i}{2}\int_{\partial D}\langle\brk{\nabla_{\frac{d}{d\theta}}V+iJV,V}\rangle\;d\theta
\end{align*}
\end{proof}
From the formula and Proposition \ref{varfield}, index form would take negative values for any section V of $f^\ast T^{(1,0)}N$ satisfying $\bbrk{V,\frac{df}{d\bar{z}}}=0$ and $\nabla_{\frac{d}{d\bar{z}}}V=0$. For those sections, the last term vanishes as it is in $f^\ast T^{(1,0)}N$. The curvature term also vanishes as the metric is flat in $C^n$. The remaning terms would be negative because of pseudoconvexity.

\section{Genus Zero with Single Boundary Component}
In this section, we prove conformality theorem for the critical points of $\overline{\partial}$-energy by changing the metric on the disk $D$. As a consequence, it gives the direction of $\frac{df}{dr}+J\frac{df}{d\theta}$ in first and second variation formula. We then give the construction of holomorphic variations and index theorem for $\overline{\partial}$-energy.
\begin{prop}\label{outward}
If $f:D\to N$ is critical to $\overline{\partial}$-energy, then $f$ is harmonic and weakly conformal. Moreover, we have boundary condition that $\frac{df}{dr}+J\frac{df}{d\theta}=\lambda\nu$ for some nonnegative function $\lambda$ on the boundary $\partial \Sigma$.
\end{prop}

\begin{proof}
By the first variation formula, if $f$ is a critical point of $\overline{\partial}$-energy, then
\[0=\int_D \brk{V,\nabla_{\frac{d}{d\bar{z}}} \frac{df}{dz}}\;dx\wedge dy+\int_{\partial D}\brk{V,\frac{df}{dr}+J\frac{df}{d\theta}}\;d\theta\]
where $V$ is a variation vector field on $\Sigma$. If we choose $V$ with compact support on $\Sigma$, then the boundary terms is zero and $\int_D\brk{V,\nabla_{\frac{d}{d\bar{z}}} \frac{df}{dz}}\;dx\wedge dy=0$ for any $V$ with compact support. This implies $\nabla_{\frac{d}{d\bar{z}}} \frac{df}{dz}=0$ on $D$ and hence $f$ is harmonic.

Now suppose that we change the metric on the disk instead of varying the map $f$. Let $g_t=\delta+t T$ be the metrics on the disk for small time $t$, where $T$ is arbitrary symmetric two tensor. For each $t$, by uniformization theorem, there exists a conformal map $h_t:(D, \delta)\to(D, g_t)$. Since the total energy is preserved by conformal mapping, we have $E(f\circ h_t, \delta)=E(f, g_t)$. On the other hand, $E(f\circ h_t, \delta)=\frac{1}{2}E''(f\circ h_t, \delta)-\frac{1}{2}\omega(N)[f(D)]$. Since the pull-back K\"{a}hler form is diffeomorphic invariance under fixed boundary, it gives
\[\frac{d}{dt}|_{t=0}E(f\circ h_t, \delta)=\frac{d}{dt}|_{t=0}\frac{1}{2}E''(f\circ h_t,\delta)+0\]
As $f$ is a critical point of $\overline{\partial}$-energy, we then have
\begin{align*}
0=\frac{d}{dt}|_{t=0}E(f\circ h_t, \delta)&=\frac{d}{dt}|_{t=0}E(f, g_t)\\
&=-\int_D tr(AT)+\frac{1}{2}tr(A)tr(T)\;dx\wedge dy
\end{align*}
where $tr(M)$ is trace of $M$ and $(A_{ij})=\begin{pmatrix} \brk{\frac{df}{dx},\frac{df}{dx}} & \brk{\frac{df}{dx},\frac{df}{dy}}\\ \brk{\frac{df}{dx},\frac{df}{dy}} & \brk{\frac{df}{dy},\frac{df}{dy}}\end{pmatrix}$. Letting $T=A$ into the formula, it give $A=\eta I$ for some positive function $\eta$. Therefore, $f$ is weakly conformal.

Next from the first variation formular, we have $\frac{df}{dr}+J\frac{df}{d\theta}=\lambda\nu$ for some function $\lambda$ on the boundary $\partial \Sigma$. We are going to show that $\lambda$ is nonnegative. Take the inner product with $\frac{df}{dr}$ on both sides, we get
\[|\frac{df}{dr}|^2-|\frac{df}{d\theta}||\frac{df}{dr}|\leq |\frac{df}{dr}|^2+\brk{J\frac{df}{d\theta},\frac{df}{dr}}=\lambda\brk{\frac{df}{dr},\nu}.\]
By the conformality of $f$, we have $|\frac{df}{d\theta}|=|\frac{df}{dr}|$. And hence
\[0\leq \lambda\brk{\frac{df}{dr},\nu}\]
As the image of $D$ is inside the region $N$, $\brk{\frac{df}{dr},\nu}$ is non-negative. Therefore, $\lambda$ is also non-negative.
\end{proof}

The index estimates is based on the construction of (1,0) holomorphic sections. Any such sections would give negative values on the index form. The construction is divided into two part. The first one is the existence of at least $2n$ holomorphic sections which are real on the boundary. We use the technique in \cite{Fraser}
to show the existence considering a system of partial differential equations. The difference is that we consider the pullback tangent bundle $f^\ast TN\otimes \C$ instead of the normal bundle so that $2n$ would be the lower bound. The second part is the construction of (1,0) holomorphic section using those $2n$ holomorphic section.

\subsection{Holomorphic sections}
The bundle $f^\ast(TN)$ over $D$ is topologically trival and, therefore, admit global basis $s_1,\ldots, s_{2n}$, where $n=dim_{\C}N$. Any section of $f^\ast{TN}\otimes \C$ can be written as $W=\sum_{i=1}^{2n} f^i s_i$, where $f^i$ are complex valued function. By Malgrange theorem, $f^\ast{TN}\otimes \C$ admit a holomorphic structure and $\overline{\partial} W= \nabla_{\overline{z}} W d\overline{z}$.

Let $\overline{\partial}s_i=a_{ij}s_j$, where $a_{ij}$ is (0,1) form. We are looking for section $W$ satisfies
\[\begin{array}{rl} \overline{\partial}W=0 & \text{on $\Sigma$} \\ Im\;W=0 & \text{on $\partial \Sigma$}\end{array}\]
The corresponding system of partial differential equations is 
\[\begin{array}{rl} \overline{\partial}f^i+\displaystyle\sum_{j=1}^{2n} f^ja_{ji}=0 & \text{for all $i$ on $D$} \\ Im\;f^i=0 & \text{on $\partial D$}\end{array}\]

Consider the operator $A:\mathcal{V}\to L^2_{(0,1)}(\Sigma)$ defined by
\[ (Af)_i=\overline{\partial}f^i+\displaystyle\sum_{j=1}^{2n} a_{ji}f^j\]
where $\mathcal{V}=\{f\in H^1(\Sigma,\C^n)\;|\;Im\;f=0\;\;\text{on $\partial \Sigma$}\}$ and $L^2_{(0,1)}(\Sigma)$ is the space of $(0,1)$-form on $\Sigma$ in $L^2$ class. Any element in the kernel of $A$ corresponds to holomorphic section that is real on the boundary and we have the following lemma
in estimating the dimension of solutions space.
\begin{lemma}{(cf. \cite{Fraser2})}\label{holosection}
The operator $A$ is Fredholm operator and the index of $A$ is $2n$. It follows that the dimension of $\{W\in\Gamma(f^\ast(TN)\otimes\C)\; : \;\overline{\partial}W=0 \;\;\text{on $\Sigma$}, Im(W)=0\;\;\text{on $\partial \Sigma$}\}$ is at least $2n$ over $\R$.
\end{lemma}

\begin{proof}
It follows from the index estimation of the operator $A$ in (\cite{Fraser},\cite{Fraser2})
. In our case, we consider the domain in $H^1(\Sigma, \C^n)$ instead of $H^1(\Sigma, \C^{n-2})$
\end{proof}

\subsection{Index Estimates}
We now prove the index theorem using those holomorphic sections in Lemma \ref{holosection}. Let $W_1, \ldots, W_{2n}$ be linearly independent solution over $\R$ to lemma \ref{holosection}. Since $(W_i, W_j)$ is holomorphic function on $\Sigma$ and real on the boundary, it is real constant over $\Sigma$. Moreover, if $W_i$ is zero on the boundary, $W_i$ is zero over $\Sigma$ as it is holomorphic. Apply rescaling and rotation if necessary, we  then select a subset $\{W_{i_1},\ldots, W_{i_n}\}$ of $\{W_1, \ldots, W_{2n}\}$ such that $\{W_{i_1}, JW_{i_1}\ldots, W_{i_n}, JW_{i_n}\}$ is a orthonormal basis of $f^\ast TN$ on the boundary

\begin{lemma}\label{holosection2}
Let $V_{i_j}=W_{i_j} -iJW_{i_j}$. Then we have the following:
\begin{enumerate}[label=\arabic*.]
\item $V_{i_j}$ is non vanishing except finitely many points.

\item $V_{i_1},\ldots, V_{i_n}$ are linearly independent over $\C$.
\end{enumerate}
\end{lemma}

\begin{proof}
By the construction, $\nabla_{\frac{d}{d\bar{z}}}V_{i_j}=\nabla_{\frac{d}{d\bar{z}}}W_{i_j}-iJ\nabla_{\frac{d}{d\bar{z}}}W_{i_j}=0$, so it is holomorphic section over $\Sigma$.
\begin{enumerate}[label=\arabic*.]
\item $V_{i_j}\neq0$ on the boundary since $W_{i_j}$ is real and nonzero on the boundary. Therefore, $V_{i_j}$ only has finitely many zero points over $\Sigma$ since it is holomorphic by the construction.
\item Suppose $\sum_{k=1}^n c_kV_{i_k}=0$, $c_k=u_k+iv_k$ are complex numbers. Then on the boundary,
\[\sum_{k=1}^n(u_k+iv_k)(W_{i_k}-iJW_{i_k})=0\]
Considering the real part of the equation, we have
\[u_kW_{i_k}+v_kJW_{i_k}=0.\]
We must have $u_k=v_k=0$ for all $k$ since  $\{W_{i_1}, JW_{i_1}\ldots, W_{i_n}, JW_{i_n}\}$ is basis on the boundary
\end{enumerate}
\end{proof}

Now we are ready to prove the index theorem.

\begin{proof}[Proof of theorem \ref{maintheorem}]
By Lemma \eqref{holosection2}, we have a linearly independent set $\{V_{i_1},\ldots, V_{i_n}\}$ over $\C$ and each $V_{i_j}$ is holomorphic section of the holomorphic tangent bundle. By the assumption $f$ is not holomorphic, then $\frac{df}{d\bar{z}}^{(1,0)}$ is not vanishing over $\Sigma$. We have $\bbrk{V_{i_k},\frac{df}{d\bar{z}}}$ not identically zero for some $k$ since $\{W_{i_1}, JW_{i_1}\ldots, W_{i_n}, JW_{i_n}\}$ is a orthonormal basis on the boundary. Then we define $U_j=\bbrk{V_{i_k},\frac{df}{d\bar{z}}}V_{i_j}-\bbrk{V_{i_j},\frac{df}{d\bar{z}}}V_{i_k}$ for each $j\neq k$. $U_j$ is orthogonal to $\frac{df}{d\bar{z}}$ with respect to $\bbrk{\;,\;}$, so Re$(U_j)$ and Im$(U_j)$ are variation vector fields by Proposition \eqref{varfield}. And by lemma \ref{indexformula},
\begin{align*}
I(U_j,U_j)&=2\int_D\|\nabla_{\frac{d}{d\bar{z}}}U_j\|^2-\langle\brk{R(U_j,\frac{df}{dz})\frac{df}{d\bar{z}},U_j}\rangle dx\wedge dy\\
&\;\;\;+\frac{1}{2}\int_{\partial D}\bbrk{\nabla_{U_j} \overline{U_j}, \frac{df}{dr}+J\frac{df}{d\theta}}d\theta-\frac{i}{2}\int_{\partial D}\langle\brk{\nabla_{\frac{d}{d\theta}}U_j+iJU_j,U_j}\rangle\;d\theta
\end{align*}
As $\nabla_{\frac{d}{d\bar{z}}}U_j=0$ and $U_j$ is a section of the holomorphic tangent bundle, the only nonvanishing term is $\bbrk{\nabla_{U_j} \overline{U_j}, \frac{df}{dr}+J\frac{df}{d\theta}}$. By proposition \ref{outward}, $\frac{df}{dr}+J\frac{df}{d\theta}=\lambda\nu$ for some nonnegative funtion $\lambda$ on the boundary and hence we have $\bbrk{\nabla_{U_j} \overline{U_j}, \frac{df}{dr}+J\frac{df}{d\theta}}<0$ by the pseudoconvexity. Then 
$0>I(U_j, U_j)=I(\text{Re}(U_j),\text{Re}(U_j) )+I(\text{Im}(U_j),\text{Im}(U_j))$. Hence either Re$(U_j)$ or Im$(U_j)$ would give negative value in index form $I(\cdot,\cdot)$. Moreover, $\{V_{i_1},\ldots, V_{i_n}\}$ is linerly independent over $\C$, so is $\{U_j\}_{j\neq k}$. Therefore, the index is at least n-1.
\end{proof}

\section{Extendsion to K\"{a}hler manifold with positive bisectional curvature, $k$-pseudoconvex}

\subsection{$k$-pseudoconvex}
A compact hypersurface $M$ in Riemannian manifold is $k$-convex if for each point $p\in M$, the trace of the second fundamental form $\Pi$ with respect to inward normal restricted to any $k$ plane of $T_pM$ is positive, i.e.
\[Tr_{U}(\Pi)>0\]
for any $k$-dimensional subspace $U$ of $T_pM$. We can define $k$-pseudoconvex in an analogous way. Suppose now that $\Omega=\{z\in C^n : \rho(z)<0\}$ is a compact domain with $C^2$ boundary in $\C^n$ or K\"{a}hler manifold with positive bisectional curvature. For each point $p$ on the boundary, we denote $W_p$ the subspace of holomorphic tangent space such that $V(\rho)=0$ for any $V\in W$. The second fundamental form $\Pi$ extended linearly give a hermitian form on each $W_p$
\[\Pi(X,Y)=-\langle\nabla_X\overline{Y},\nu\rangle\]
where $\nu$ is unit outward normal to $\partial \Omega$.

\begin{definition}{($k$-pseudoconvexity)}
The domain $\Omega$ is strictly $k$-pseudoconvex if for every point $p\in\partial\Omega$,
\[Tr_{U}(\Pi)>0\]
for any $k$-dimensional subspace $U$ of $W_p$
\end{definition}

In the construction of holomorphic section, we obtain $n-1$ holomorphic variations. And that provides a general index theorem in $k$-pseudoconvex domain.
\begin{theorem}
Let $D$ be unit disk and $N$ be strictly $k$-pseudoconvex domain in $ C^n$. Let $f:D\to N$ be a smooth map, with $f(\partial D)\subset \partial N$. Suppose $f$ is critical point of $\overline{\partial}$-energy. If $f$ is not holomorphic, then $f$ has index at least $n-k$ for the $\overline{\partial}$-energy.
\end{theorem}
For the case $k=1$, the domain is pseudoconvex and the above theorem coincides with theorem \ref{maintheorem}.

\begin{proof}
We define $U_1,\ldots, U_{n-1}$ as the holomorphic variations in the proof of theorem \ref{maintheorem}. As $\{U_1,\ldots, U_{n-1}\}$ is linearly independent over $\C$, for any point $p$ on $\partial \Omega$, any $\{U_{i_1},\ldots,U_{i_k}\}$ with $1\leq i_1<\ldots<i_k\leq n$ would span a $k$-dimensional subspace of $W_p$. Then by $k$-pseudoconvex,
\[\sum_{j=1}^{k}I(U_{i_j},U_{i_j})=\sum_{j=1}^k\int_{\partial D}\bbrk{\nabla_{U_{i_j}}\overline{U_{i_j}},\nu}<0\]
Therefore, there is at least one $j$, $1\leq j \leq k$ such that $I(U_{i_j},U_{i_j})<0$ and the index would be at least $n-k$.
\end{proof}

\subsection{Positive bisectional curvature}
In $\C^n$, the metric is flat and the curvature vanishes in the formula of index form. However, We can extend the theorem \ref{maintheorem} to a K\"{a}hler manifold $M$ with nonnegative holomorphic bisectional curvature,

\begin{definition}{(Holomorphic Bisectional Curvature)}
Let $X$, $Y$ be two unit tangent vector at a point in $M$. The holomorphic bisectional curvature is defined as
\[BHR(X,Y)=R(JX,X,Y,JY)=R(\xi, \overline{\eta}, \eta, \overline{\xi})\]
where $\xi=\frac{\sqrt{2}}{2}(X-iJX)$ and $\eta=\frac{\sqrt{2}}{2}(Y-iJY)$.
\end{definition}

From the formula in lemma \ref{indexformula}, the curvature term is holomorphic bisectional curvature when $V$ is a section of $f^\ast T^{(1,0)}N$. And the construction of (1,0) holomorphic variation can be done on any K\"{a}hler manifold. Therefore, we have index theorem on pseudoconvex domain in K\"{a}hler manifold with nonnegative holomorphic bisectional curvature.

\begin{theorem}
Let $D$ be a disk and $N$ be a strictly pseudoconvex domain in $n$-dimension K\"{a}hler manifold with nonnegative bisectional curvature. Let $f:D\to N$ be a smooth map, with $f(\partial D)\subset \partial N$. Suppose $f$ is critical point of $\overline{\partial}$-energy. If $f$ is not holomorphic, then $f$ has index at least n-1 for the $\overline{\partial}$-energy.
\end{theorem}

\begin{proof}
Let $U_j$ be the holomorphic variations constructed in the proof of theorem \ref{maintheorem}. Then by lemma \ref{indexformula},
\begin{align*}
I(U_j,U_j)&=2\int_D-\langle\brk{R(U_j,\frac{df}{dz})\frac{df}{d\bar{z}},U_j}\rangle dx\wedge dy+\frac{1}{2}\int_{\partial D}\bbrk{\nabla_{U_j} \overline{U_j}, \frac{df}{dr}+J\frac{df}{d\theta}}d\theta
\end{align*}
Since $U_j\in\Gamma(f^\ast(T^{(1,0)}N))$, $\langle\brk{R(U_j,\frac{df}{dz})\frac{df}{d\bar{z}},U_j}\rangle=\langle\brk{R(U_j,\frac{df}{dz}^{(0,1)})\frac{df}{d\bar{z}}^{(1,0)},U_j}\rangle$ which is holomorphic bisectional curvature and hence nonnegative. Therefore, $I(U_j,U_j)<0$.
\end{proof}

\section{Minimal disks in pseudoconvex domain in $\C^2$}
In this section, we discuss examples of free boundary minimal disk for energy and $\overline{\partial}$-energy in pseudoconvex domains in $\C^2$. We identify $\C^2=(z_1, z_2)$ as $\R^4=(x_1, x_2, y_1, y_2)$ with the usual almost complex structure $J$ such that $J\frac{d}{dx^i}=\frac{d}{dy^i}$ for $i=1$, $2$. The domain $M=\{(x_1,x_2,y_1,y_2)\;|\; x_1^2+x_2^2<1\}$ is a strictly pseudoconvex domain in $\C^2$. The first example is a stable disk for energy but not for $\overline{\partial}$-energy. Let $f_1:D\to M$ be a map defined by 
\[f_1(x,y)=(x , y, 0, 0)\]
$f_1$ does not minimize the $\overline{\partial}$-energy as there is a continuous deformation of $f_1$ to a holomorphic disk $f_2(x,y)=(x,y,y,-x)$ which is given by $F_t(x,y)=(x,y,ty,-tx)$ and the $\overline{\partial}$-energy is decreasing as $t$ increases. That deformation also shows that $f_1$ is not even a critical point of $\overline{\partial}$-energy. To see that $f_1$ is a stable free boundary minimal disk, consider the projection $\pi$ to the first two coordinate $(x_1, x_2)$. Any surface homotopic to $f_1$ in $M$ would be mapped to a unit disk in $\R^2$ and the energy of $f_1$ is same as the area of the unit disk. Since the image has smaller area under the projection and the energy is always greater than or equal to the area, $f_1$ therefore minimizes the energy. A holomorphic disk in pseudoconvex domain does not necessarily minimize the energy. As mentioned above, $f_2$ is a holomorphic disk and hence minimizes $\overline{\partial}$-energy, but $f_2$ does not minimize the energy as it has larger area than $f_1$.

The next example is a critical point for $\overline{\partial}$-energy that is not holomorphic. Let $M=\{(x_1,x_2,y_1,y_2)\;|\; x_1^2+x_2^2+y_1^2+y_2^2<1\}$ be a unit ball in $\R^4$. It is a strictly pseudoconvex domain since it is convex. Let $f_3:D\to M$ be defined by 
\[f_3(x,y)=(x, 0,-y,0)\]
On the boundary of the disk, the outward unit normal to $\partial M$ at $(x, 0, -y, 0)$ is $(x,0, -y, 0)$. And
\[\frac{df_3}{dr}+J\frac{df_3}{d\theta}=2(x, 0, -y, 0)\]
which is parallel to the outward unit normal. Moreover, $f_3$ is harmonic. By the first variation formula of $\overline{\partial}$-energy, $f_3$ is a critical point of $\overline{\partial}$-energy. But $f_3$ is not holomorphic and also not minimizing the $\overline{\partial}$-energy. Moreover, $f_3$ is unstable for both energy and $\overline{\partial}$-energy.

Finally, we will give an example of a stable free boundary disk for $\overline{\partial}$-energy in a weakly pseudoconvex domain. It illustrates that strictly pseudoconvex condition is essential to the instability of critical points of $\overline{\partial}$-energy. Let $M$ be the domain $\{(x_1, x_2, y_1, y_2) | (x_1-y_2)^2+(x_2+y_1)^2<1\}$ which is weakly pseudoconvex and $f_4 : D\to M$ be a map defined by $f_4(x,y)=(x,-y,0,0)$. Since $\frac{df_4}{dr}+J\frac{df_4}{d\theta}=(x, -y,-y, -x)$ that is parallel to the outward normal to $\partial M$ and $f_4$ is harmonic, $f_4$ is a critical point of $\overline{\partial}$-energy. We will prove the stability of $f_4$ in $\overline{\partial}$-energy by starting with construction of deformation of $f_4$.
Let $\sigma$, $\varphi$, $\psi$, $\eta$ be smooth function on $D$ with $\sigma=0$ on $\partial D$. Using spherical coordinate $x=r\cos \theta$ , $y=r\sin\theta$, we define a continuous deformation of $f_4$ by
\begin{align*}
F_t(r,\theta)&=\Big(\frac{(1+t\sigma)r\cos(\theta-t\varphi)+r\cos\theta}{2}+t\psi, \frac{-(1+t\sigma)r\sin(\theta-t\varphi)-r\sin\theta}{2}+t\eta,\\
&\quad \frac{-(1+t\sigma)r\sin(\theta-t\varphi)+r\sin\theta}{2}-t\eta, \frac{-(1+t\sigma)r\cos(\theta-t\varphi)+r\cos\theta}{2}+t\psi\Big)
\end{align*}
Then $\frac{dF_t}{dt}|_{t=0}=\frac{\sigma}{2}(x,-y,-y,-x)+\frac{\varphi}{2}(y,x,x,-y)+\psi(1,0,0,1)+\eta(0,1,-1,0)$. As  $\{(y,x,x,-y), (1,0,0,1), (0,1,-1,0)\}$ are orthogonal basis to tangent space of $\partial M$ at the boundary of the disk and $\sigma=0$ on $\partial D$, $F_t$ is a valid variation of $f_4$ for small $t$. The $\overline{\partial}$-energy of $F_t$ is $\int_D|\overline{\partial}F_t|^2=\int_D|\frac{dF_t}{dx}+J\frac{dF_t}{dy}|^2$. Using $\frac{d}{dx}=\cos\theta\frac{d}{dr}-\frac{1}{r}\sin\theta\frac{d}{d\theta}$ and $\frac{d}{dy}=\sin\theta\frac{d}{dr}+\frac{1}{r}\cos\theta\frac{d}{d\theta}$, we then have
\[\frac{dF_t}{dx}=\Big(\frac{1+A}{2}+t\psi_x, \frac{B}{2}+t\eta_x, \frac{B}{2}-t\eta_x, \frac{1-A}{2}+t\psi_x\Big)\]
and 
\[\frac{dF_t}{dy}=\Big(\frac{C}{2}+t\psi_y,\frac{D-1}{2}+t\eta_y, \frac{D+1}{2}-t\eta_y, \frac{-C}{2}+t\psi_y\Big)\]
where
\begin{align*}
A&=(1+t\sigma)\cos(t\varphi)+(1+t\sigma)rt\cos\theta\sin(\theta-t\varphi)\frac{d\varphi}{dr}+rt\cos\theta\cos(\theta-t\varphi)\frac{d\sigma}{dr}\\
&\quad-(1+t\sigma)t\sin\theta\sin(\theta-t\varphi)\frac{d\varphi}{d\theta}-t\sin\theta\cos(\theta-t\varphi)\frac{d\sigma}{d\theta}\\
B&=(1+t\sigma)\sin(t\varphi)+(1+t\sigma)rt\cos\theta\cos(\theta-t\varphi)\frac{d\varphi}{dr}-rt\cos\theta\sin(\theta-t\varphi)\frac{d\sigma}{dr}\\
&\quad-(1+t\sigma)t\sin\theta\cos(\theta-t\varphi)\frac{d\varphi}{d\theta}+t\sin\theta\sin(\theta-t\varphi)\frac{d\sigma}{d\theta}\\
C&=(1+t\sigma)\sin(t\varphi)+(1+t\sigma)rt\sin\theta\sin(\theta-t\varphi)\frac{d\varphi}{dr}+rt\sin\theta\cos(\theta-t\varphi)\frac{d\sigma}{dr}\\
&\quad+(1+t\sigma)t\cos\theta\sin(\theta-t\varphi)\frac{d\varphi}{d\theta}+t\cos\theta\cos(\theta-t\varphi)\frac{d\sigma}{d\theta}\\
D&=-(1+t\sigma)\cos(t\varphi)+(1+t\sigma)rt\sin\theta\cos(\theta-t\varphi)\frac{d\varphi}{dr}-rt\sin\theta\sin(\theta-t\varphi)\frac{d\sigma}{dr}\\
&\quad+(1+t\sigma)t\cos\theta\cos(\theta-t\varphi)\frac{d\varphi}{d\theta}-t\cos\theta\sin(\theta-t\varphi)\frac{d\sigma}{d\theta}
\end{align*}
Therefore,
\[\frac{dF_t}{dx}+J\frac{dF_t}{dy}=\Big(\frac{E}{2}+t\psi_x+t\eta_y, \frac{F}{2}+t\eta_x-t\psi_y, \frac{F}{2}-t\eta_x+t\psi_y, \frac{-E}{2}+t\psi_x+t\eta_y\Big)\]
where
\begin{align*}
E&=2(1+t\sigma)\cos(t\varphi)-(1+t\sigma)rt\sin(t\varphi)\frac{d\varphi}{dr}+rt\cos(t\varphi)\frac{d\sigma}{dr}-(1+t\sigma)t\cos(t\varphi)\frac{d\varphi}{d\theta}\\
&\quad-t\sin(t\varphi)\frac{d\sigma}{d\theta}\\
F&=2(1+t\sigma)\sin(t\varphi)-(1+t\sigma)rt\cos(t\varphi)\frac{d\varphi}{dr}+rt\sin(t\varphi)\frac{d\sigma}{dr}-(1+t\sigma)t\sin(t\varphi)\frac{d\varphi}{d\theta}\\
&\quad+t\cos(t\varphi)\frac{d\sigma}{d\theta}\\
\end{align*}
Then 
\begin{align*}
\Big|\frac{dF_t}{dx}+J\frac{dF_t}{dy}\Big|^2&=(\frac{E}{2}+t\psi_x+t\eta_y)^2+(\frac{F}{2}+t\eta_x-t\psi_y)^2+(\frac{F}{2}-t\eta_x+t\psi_y)^2\\
&\;\quad+(\frac{-E}{2}+t\psi_x+t\eta_y)^2\\
&=2\Big((\frac{E}{2})^2+(\frac{F}{2})^2+(t\psi_x+t\eta_y)^2+(t\eta_x-t\psi_y)^2\Big)\\
&=\frac{1}{2}(E^2+F^2)+2t^2(\psi_x+\eta_y)^2+2t^2(\eta_x-\psi_y)^2
\end{align*}
And
\begin{align*}
E^2+F^2&=4+8t\sigma+4t^2\sigma^2+4(1+t\sigma)rt\frac{d\sigma}{dr}-4(1+t\sigma)^2t\frac{d\varphi}{d\theta}+((1+t\sigma)rt\frac{d\varphi}{dr}+t\frac{d\sigma}{d\theta})^2\\
&\quad+(rt\frac{d\sigma}{dr}-(1+t\sigma)t\frac{d\varphi}{d\theta})^2
\end{align*}
Since 
\begin{align*}
\int_D 8t\sigma+4t^2\sigma^2+4(1+t\sigma)rt\frac{d\sigma}{dr}&=\int_0^{2\pi}\int_0^1 8rt\sigma+4rt^2\sigma^2+4(1+t\sigma)r^2t\frac{d\sigma}{dr}\;\;drd\theta\\
&=\int_0^{2\pi}\int_0^1 \frac{d}{dr}(4r^2t\sigma+2r^2t^2\sigma^2)\;\;drd\theta\\
&=0,
\end{align*}
we have
\begin{align*}
\int_D|\frac{dF_t}{dx}+J\frac{dF_t}{dy}|^2=&\int_D\frac{1}{2}\Big(4-4(1+t\sigma)^2t\frac{d\varphi}{d\theta}+t^2((1+t\sigma)r\frac{d\varphi}{dr}+\frac{d\sigma}{d\theta})^2\\
&\quad+t^2(r\frac{d\sigma}{dr}-(1+t\sigma)\frac{d\varphi}{d\theta})^2\Big)+2t^2(\psi_x+\eta_y)^2+2t^2(\eta_x-\psi_y)^2
\end{align*}
and $\frac{d}{dt}\Big|_{t=0}\int_D|\frac{dF_t}{dx}+J\frac{dF_t}{dy}|^2=0$ which agrees the fact that $f_4$ is critical point of $\overline{\partial}$-energy. We then compute the second derivative of $\overline{\partial}$-energy at $t=0$,
\begin{align*}
\frac{d^2}{dt^2}\Big|_{t=0}\int_D |\frac{dF_t}{dx}+J\frac{dF_t}{dy}|^2=&\int_D -8\sigma\frac{d\varphi}{d\theta}+(r\frac{d\varphi}{dr}+\frac{d\sigma}{d\theta})^2+(r\frac{d\sigma}{dr}-\frac{d\varphi}{d\theta})^2\\
&\qquad+4(\psi_x+\eta_y)^2+4(\eta_x-\psi_y)^2
\end{align*}
The first term on the right hand side can be rewritten as follow
\begin{align*}
\int_D -8\sigma\frac{d\varphi}{d\theta}&=\int_0^1\int_0^{2\pi}-8r\sigma\frac{d\varphi}{d\theta}\;dr\;d\theta\\
&=\int_0^1\int_0^{2\pi}\frac{d}{dr}(-4r^2\sigma\frac{d\varphi}{d\theta})+4r^2\frac{d\sigma}{dr}\frac{d\varphi}{d\theta}+4r^2\sigma\frac{d^2\varphi}{drd\theta}\;drd\theta\\
&=\int_0^1\int_0^{2\pi}4r^2\frac{d\sigma}{dr}\frac{d\varphi}{d\theta}-4r^2\frac{d\sigma}{d\theta}\frac{d\varphi}{dr}\;drd\theta
\end{align*}
Therefore,
\begin{align*}
\frac{d^2}{dt^2}\Big|_{t=0}\int_D |\frac{dF_t}{dx}+J\frac{dF_t}{dy}|^2&=\int_D (r\frac{d\varphi}{dr}-\frac{d\sigma}{d\theta})^2+(r\frac{d\sigma}{dr}+\frac{d\varphi}{d\theta})^2\\
&\qquad +4(\psi_x+\eta_y)^2+4(\eta_x-\psi_y)^2\\
&\geq0
\end{align*}
Hence $f_4$ is stable for any variation $V=\frac{\sigma}{2}(x,-y,-y,-x)+\frac{\varphi}{2}(y,x,x,-y)+\psi(1,0,0,1)+\eta(0,1,-1,0)$. Now we are going to show that $f_4$ is stable to any arbitrary variation $V$. Suppose that $f_4$ is unstable to a variation $V$. Let $\epsilon>0$, $\rho_\epsilon$ be a radical symmetric smooth function on the disk $D$ such that $\rho_\epsilon=0$ on $B_{\epsilon^2}(0)$, $\rho_\epsilon=1$ on $D\backslash B_\epsilon(0)$ and $|\frac{d\rho_\epsilon}{dr}|\leq \frac{1}{r|\ln \epsilon|}$. Then $f_4$ would be stable to the variation $\rho_\epsilon V$. From the second variation formula of $\overline{\partial}$-energy given by the index form,
\begin{align*}
0\leq I(\rho_\epsilon V, \rho_\epsilon V)=&\frac{1}{2}\Big[\int_D \langle\nabla_{\frac{d}{dx}}(\rho_\epsilon V),\nabla_{\frac{d}{dx}}(\rho_\epsilon V)\rangle+\langle\nabla_{\frac{d}{dy}}(\rho_\epsilon V), \nabla_{\frac{d}{dy}}(\rho_\epsilon V)\rangle dx\wedge dy\\
&\;\;+\int_{\partial D}\langle\nabla_V V,\frac{df}{dr}+J\frac{df}{d\theta}\rangle + \langle J\nabla_{\frac{d}{d\theta}}V, V\rangle\;d\theta\Big]\\
=&\frac{1}{2}\Big[\int_D \|\nabla \rho_\epsilon\|^2\|V\|^2+2\rho_\epsilon\frac{d\rho_\epsilon}{dx}\langle V, \nabla_\frac{d}{dx} V\rangle+\\
&\;\qquad+2\rho_\epsilon\frac{d\rho_\epsilon}{dy}\langle V, \nabla_\frac{d}{dy} V\rangle+\rho_\epsilon^2\|\nabla V\|^2\;dx\wedge dy\\
&\;\;\;+\int_{\partial D}\langle\nabla_V V,\frac{df}{dr}+J\frac{df}{d\theta}\rangle + \langle J\nabla_{\frac{d}{d\theta}}V, V\rangle\;d\theta\Big]\\
\leq&\frac{1}{2}\Big[\int_D \|\nabla\rho_\epsilon\|^2\|V\|^2+4\rho_\epsilon\|\nabla \rho_\epsilon\|\|V\|\|\nabla V\|\;dx\wedge dy\Big]+I(V,V)\\
\leq&C\Big[\int_D \|\nabla\rho_\epsilon\|^2+4\rho_\epsilon\|\nabla \rho_\epsilon\|\Big]+I(V,V)
\end{align*}
where the constant $C$ depends on $\|V\|$, $\|\nabla V\|$. Then
\begin{align*}
\int_D \|\nabla \rho_\epsilon\|^2=&\int_0^{2\pi}\int_0^1\|\nabla \rho_\epsilon\|^2r\;drd\theta\\
\leq&\int_0^{2\pi}\int_{\epsilon^2}^{\epsilon}\frac{1}{r|\ln\epsilon|^2}\;drd\theta\\
\leq&\frac{C}{|\ln\epsilon|}
\end{align*}
and
\begin{align*}
\int_D 4\rho_\epsilon \|\nabla\rho_\epsilon\|\leq&\int_0^{2\pi}\int_{\epsilon^2}^\epsilon \frac{C}{|\ln\epsilon|}\;drd\theta\\
\leq&\frac{C\epsilon}{|\ln\epsilon|}
\end{align*}
Therefore, 
\[0\leq C(\frac{1}{|\ln\epsilon|}+\frac{\epsilon}{|\ln\epsilon|})+I(V,V)\]
Letting $\epsilon$ to 0, we have $I(V,V)\geq0$ and contradiction aries. Therefore $f_4$ is a stable critical point of $\overline{\partial}$-energy.

 \bibliographystyle{amsplain}
 
\end{document}